\documentclass[journal]{IEEEtran}

\usepackage{amssymb}
\usepackage{amsmath}
\usepackage{amsthm}
\usepackage{epsfig}
\usepackage{graphicx}
\usepackage{amsfonts}
\usepackage{subfigure}

\newtheorem{theorem}{Theorem}

\newtheorem{lemma}{Lemma}

\newtheorem{proposition}{Proposition}

\begin{document}

\title{Towards Green Shipping with Integrated Fleet Deployment and Bunker Management}

\author{Yanyan~Tong,
        Jianfeng~Mao~
        \thanks{Yanyan Tong and Jianfeng Mao are with the Division of Systems Engineering and Management, Nanyang Technological University, Singapore 639798, e-mail: ytong1@e.ntu.edu.sg, jfmao@ntu.edu.sg.}
        }

\maketitle

\begin{abstract}
The vast majority of world trade is carried by the sea shipping industry and petroleum is still the major energy source. To reduce greenhouse gas emissions, it is an urgent need for both industry and government to promote green shipping. To move shipping towards a more environmentally and financially sustainable future, we consider to develop an integrated fleet deployment and bunker management model in this paper. The fleet deployment refers to the decisions on which ships to operate on which route, at what speed and what cargos to be transported. The bunker management refers to the decision on which ships to bunker on which port and at what volume. A more realistic fuel consumption function with considering not only cruise speed but also freight tonnage on board is adopted in the proposed optimization model, which can further improve its effectiveness in practice. An efficient algorithm termed ``Most Promising Area Search'' (MPAS) is developed afterwards to solve the corresponding mixed integer programming problem, which can guarantee to converge to a local optimal solution. The benefits of the integrated model and the computational performance of the MPAS are illustrated by numerical examples and a case study.
\end{abstract}


\section{Introduction}

The rapid growth of world population and people's living standard requires international exchanges of resources, which increases the dependence of world economy on international trading. Shipping companies provide transportation between distant markets, with the cost of operating the vessel fleet as a central concern.

Short sea shipping is the main provider for domestic trading of coastwise countries and international shipping between short-distance countries, especially in Europe and southeast Asia. For example, Coastal shipping in Greece constitutes the most important means of transporting goods and the primary means of passenger transportation, due to the geographic discontinuity\cite{Sambracos04}.

When environmental concerns have gained much attentions recently, reducing oil consumption and greenhouse gas emissions has been a key issue in all industries. Shipping heavily replies on oil, and is not yet in a position to adopt alternative energy sources. The rapid development in global trading over the past decades was powered by easily available and affordable oil. The fuel cost contributes to a large portion of the total operational cost of a shipping company. At the price of 135 USD per ton, bunker fuel costs constituted about half the operating cost of a large containership company \cite{Notteboom06}. And when fuel price reaches 500 USD per ton, bunker fuel costs will be as much as about three quarters of the operating cost of a large containership company \cite{Ronen11}. As evidenced by surges in oil prices these years, one main challenge faced by shipping companies is how to reduce fuel cost in operations. Appropriate marine fuel management strategy will not only save the environment by less greenhouse gas emission, but also benefit shipping companies by less operational costs.

Sea shipping operators have adopted certain ways to implement fleet deployment to determine appropriate route and speed for each ship to reduce fuel consumption and associated financial hemorrhage \cite{Leach08} (laying up or chartering out ships are also viable options when facing insufficient demands). Generally, slow steaming saves fuel consumption while also releasing fewer greenhouse gas. However, lower speed may lengthen the traveling time at the risk of violating service schedule. Hence, it is necessary to carefully tune cruise speed for each ship on each voyage to trade off operational cost and service quality. On the other hand, fuel prices vary a lot at different ports and sea shipping operators can take this advantage to further enhance the profitability through bunker management, i.e., where and how much to bunker, which is also critical to reduce the total operational cost.

In general, the fleet deployment refers to the decisions on which ships to operate on which route, at what speed and what cargos to be transported. The bunker management refers to the decision on which ships to bunker on which port and at what volume. And the operational cost are mainly influenced by both fleet deployment and bunker management. These two types of decisions are separately considered in practice and most of literatures, which is quite helpful in reducing the complexity in decision making. However, they are tightly coupled by planned routes. Separately considering them may not achieve efficient solutions. To further enhance green shipping, we will develop an integrated fleet deployment and bunker management model to derive better solutions.

The fuel consumption estimation itself is another crucial issue in this problem, which is mainly related to the voyage distance and the resistance a ship overcomes in water. Theoretically, not only cruise speed but also displacement contribute to the resistance. The resistance a ship overcomes is proportional to the wetted surface and the square of cruise speed \cite{PracticalShipDesign}. The wetted surface is directly related to the displacement, and hence the total weight of the ship. The fuel consumption rate is proportional to the power of ship, which can be expressed as $Fv$, and hence is proportional to the displacement and $v^3$. As the deadweight of a large cargo ship can be several times the lightweight, the effect of load on fuel consumption rate should not be ignored. However, the factor of displacement is ignored or assumes to be constant throughout voyages in most of literatures. This approximation may result in an inferior solution and degrade the effectiveness in practice. Since the pickup and delivery service order may result in quite fluctuating load on board, it is important and beneficial to consider this effect on fuel consumption from the changing of weights on board when deciding shipping service orders. We will adopt a more realistic fuel consumption function in the model.

The paper is organized as following. In the next section, related works are reviewed. Then we will formulate the problem as a mixed integer programming in section 3. Two heuristic algorithms are proposed in section 4, one of which is based on the concept of Most Promising Area and is proved to converge with probability 1. Numerical experiments are carried out in section 5 to illustrate the necessity of joint optimization and the performances of heuristics. We will close with conclusions presented in section 6.

\section{Related Works}

As early as 1982, Ronen \cite{Ronen82} presented three models for determining the optimal cruise speed under different commercial circumstances. He also pointed out that the fuel consumption is approximately proportional to the third power of cruise speed for a single vessel on an empty voyage leg. However, a review chapter by Christiansen et al.\cite{Christiansen07} reveals that relatively little research has been done regarding optimizing operation cost be adjusting the speed of vessels.

Alvarez \cite{Alvarez09} presented a model to address the routing problem and deployment problem jointly. Ship operating speed selection is considered when generating parameters for the fuel cost in objective function. The third power relationship is assumed in the model. The speed cannot be optimized since it is considered as a parameter rather than a decision variable.
Yao et al.\cite{Yao12} considered a bunker fuel management strategy study for a single ship on a liner service with fixed routing, where they highlighted the importance of joint optimization of bunkering ports selection, bunkering amount determination, and ship speed adjustment. However, the decisions are limited to bunkering and speed, with given routing information.

The fleet deployment on both strategic and tactical levels in shipping industry is highly related to the vehicle routing problems with pickups and deliveries(VRPPD), which have attracted much attention in research over the past few decades. There are generally three kinds of models in VRPPD.

\begin{enumerate}
    \item Vehicle routing problem with backhauls(VRPB): This kind is also known as delivery-first, pickup-second VRPPD, as all deliveries must be completed before commencing pick-ups. Customers can be divided into linehauls with delivery demands, and backhauls with pickup demands.
    \item Vehicle routing problem with mixed pickup and delivery(VRPMPD): No restriction is  set on the sequence of linehauls and backhauls.
    \item Vehicle routing problem with simultaneous pickup and delivery(VRPSPD): Customer ports may simultaneously receive and send goods.
\end{enumerate}

For a same problem, adopting the assumption of VRPB will often result in poor quality solutions compared with the other two \cite{Nagy05}. Although it may cause some difficulties in operation, VRPMPD and VRPSPD are more efficient ways to manage shipping activities. We note that VRPMPD can be modelled as VRPSPD with either pickup demand or delivery demand being zero. Hence, we will focus on vehicle routing problem with simultaneous pickup and delivery.

The vehicle routing problem with simultaneous pickups and deliveries was first introduced by Min\cite{Min89}, where a practical problem of a public library is solved. Customers are first clustered into groups, and TSPs are solved for each group. After penalization of infeasible arcs, TSPs are re-solved. Gendreau et al.\cite{Gendreau99} studied the traveling salesman problem with pickup and delivery by solving them as simple TSP regardless of pickup and delivery, and determine the order of pickups and deliveries based on the TSP tour. These two algorithms are developed based on their specific problem structure, and provided acceptable performance. However, their problem structure is quite different from our problem, with much less decision variables. Nagy and Salhi\cite{Nagy05} proposed a heuristic to solve both VRPSPD and VRPMPD. Initial solutions are generated based on constructive heuristic and composite improvement procedures are carried out. A lot other papers use the similar two-stage heuristics with different generation methods. Ganesh and Narendran\cite{Ganesh07} addressed the problem by a 4-phase algorithm. Two-stage heuristics can be used to solve complex problem when simple initial solutions and effective improvement procedures can be provided. Hence, a two-stage heuristic algorithm will be proposed in this paper to solve our problem.

\section{Problem Formulation}

\subsection{Problem description}

In this paper, we consider a specific sea shipping company with an depot settled in a city in east Asia. It mainly operates short sea shipping business between ports in Asia in a tramp mode through the depot.

Without loss of generality, assume that the shipping company owns a fleet of $K$ ships with limited capacities and is planning to operate shipping service on a set of ports ${1,2,...,N}$, among which port $1$ is the depot, and the other ports are customer ports. The shipping network can be modeled as a graph $G=(V,E)$, where $V={1,2,...,N}$ is the set of ports and $E$ is the set of feasible voyages between ports.

Demands for pickup and delivery may simultaneously exist for each customer port. Each unit of fulfilled demand would bring some revenue to the company. Ships start their journey from the depot, visit a sequence of customer ports, and return to the depot eventually. A ship can also be charted out (or just laid up) if the potential profit is less than some threshold due to insufficient demands or high operational cost. (The threshold can be estimated as the profit by chartering out the vessel or some operating overhead.)

As splitting demands such that the delivery of certain loads is completed in multiple trips rather than one trip results in opportunities for a reduction in cost and the number of vehicle used, by the efficient use of excess capacities \cite{Nowak08}, it is allowed in this model that the demands for a certain customer port can be fulfilled by more than one ship.

Each customer port may require some visiting time window in which ships can serve. The port processing times at each port, including entry time, unloading and loading time, idle time and exiting time, is assumed to be known or can be estimated with reasonable accuracy, and hence is assumed to be fixed in this model. The sailing time between any two ports is determined by the cruise speed on that voyage. Although the speed is practically variable along time, the average cruise speed can be utilized to determine the schedule and fuel consumption in the tactical level. So the speed is assumed to be constant within each voyage.

The company receives higher revenue by fulfilling more demands, which in turn raises fuel consumption and increases greenhouse gas emissions due to larger amount of freight tonnage on board or longer traveling distance. The fleet deployment and bunker management are applied to improve profitability and reduce operational cost through two ways respectively: less fuel consumption and lower bunkering cost. In the fleet deployment, the fuel consumption rate is proportional to the total weight and the third power of speed with the assumption that the displacement is approximately proportional to the total weight of a ship. In the bunker management, the fuel prices are different at each port. As incremental quantity price discounts are commonly applied \cite{Weng95}, and bunkering ports usually provide price discounts under contractual agreement \cite{Yao12}, it is assumed in this problem that ports provide price discounts on incremental quantity of fuel bunkered by a ship.

The shipping company tries to  maximize the total profit earned by the fleet, where the operational cost is considered as the total bunker fuel costs.  Although shipping service incurs some other costs, such as the administrative costs, capital costs, container costs, etc.,these cost can be considered approximately constants with the number of ships and total cycle time fixed.

The characteristics of the problem can be summarized as follows:
\begin{enumerate}
\item Pickup and delivery demands may simultaneously exist at every customer port;
\item The company is allowed to choose ports to serve according to demands received;
\item Partial fulfillment is allowed for each serving port;
\item Demands at each port can be split and served by multiple ships;
\item It is a joint optimization of routing, scheduling, speed, pickup and delivery quantities, bunkering ports and volume for each ship (laying up or chartering out ships are viable options);
\item The freight tonnage on board is considered as an important factor in estimating fuel consumption rate, instead of cruise speed as the only factor. This assumption leads to more precise estimation while produces nonlinear constraints;
\item When determining bunkering policies, different fuel prices at each port is taken into consideration. The incremental quantity price discounts are also considered.
\end{enumerate}

\subsection{Problem Formulation}
The decision variables the model determines are as follows:

\begin{tabular}{p{15pt}p{200pt}}
\\
$Q_{i1}^k$ & Delivery quantities from depot to port $i$ by ship $k$\\
$Q_{i2}^k$ & Pickup quantities from port $i$ to depot by ship $k$\\
$t_{ij}^k$ & Traveling time of ship $k$ per unit distance from port $i$ to $j$\\
$x_{ij}^k$ & Whether ship $k$ travels from port $i$ to $j$. \\
& Note that $x_{1,N+1}^k=1$ indicates that ships $k$ stays at depot.\\
$B_i^k$ & Bunkering volume of ship $k$ at port $i$\\
$y_i^k$ & Whether to bunker ship $k$ at port $i$\\
 &
\end{tabular}

The state variables used in the model are:

\begin{tabular}{p{15pt}p{200pt}}
\\
$I_i^k$ & Residual fuel volume of ship $k$ when entering port $i$\\
$W_i^k$ & Displacement of ship $k$ when leaving port $i$\\
$a_i^k$ & Arrival time of ship $k$ at port $i$\\
 &
\end{tabular}

The following parameters represent the known data for the problem:

\begin{tabular}{p{15pt}p{200pt}}
\\
$N$ & Number of ports including the depot.\\
&For ease of notation, $N+1$ is defined as the index for a virtual\\
&port with the same location as the depot.\\
$K$ & Number of ships\\
$R^k$ & Threshold revenue for ship $k$\\
$r_{i1}$ & Unit delivery revenue at port $i$\\
$r_{i2}$ & Unit pickup revenue at port $i$\\
$D_{i1}$ & Total delivery demands at port $i$\\
$D_{i2}$ & Total pickup demands at port $i$\\
$I_{\max}^k$ & Maximum fuel volume of ship $k$\\
$e^k$ & Minimum bunkering level of ship $k$\\
$c^k$ & Safety fuel level of ship $k$\\
$W_0^k$ & Lightweight of ship $k$, i.e. the weight of empty ship\\
$W_{\max}^k$ & Deadweight of ship $k$, i.e. the weight when fully loaded\\
$w$ & weight per unit cargo\\
$d_{ij}$ & Distance between port $i$ and $j$\\
$s_i^k$ & Port processing time of ship $k$ at port $i$\\
$M$ & A large enough number\\
 &
\end{tabular}

The problem can be formulated as the Mix Integer Programming (MIP) Problem below. The program is nonlinear because of the nonlinear fuel cost function of $f_i(B_i^k)$ and the fuel consumption function of $g_k(W_i^k,t_{ij}^k)$, which can be approximated by piece-wise linear function later.

\newcounter{mytempeqncnt}
\begin{figure*}
\normalsize
\setcounter{mytempeqncnt}{\value{equation}}
\setcounter{equation}{0}
\begin{align}
\max\quad & \sum\limits_{k=1}^K \sum\limits_{i=2}^N (r_{i1}Q_{i1}^k+r_{i2}Q_{i2}^k)+\sum\limits_{k=1}^K R_kx_{1,N+1}^k-\sum\limits_{k=1}^K \sum\limits_{i=1}^N f_i(B_i^k)\nonumber\\
\mathrm{s.t.}\quad & \sum\limits_{k=1}^K Q_{i1}^k\leq D_{i1}, \sum\limits_{k=1}^K Q_{i2}^k\leq D_{i2}, \quad \forall i=2,3,...,N;\\
& Q_{i1}^k\leq M\sum_{j\neq i}x_{ij}^k, Q_{i2}^k\leq M\sum_{j\neq i}x_{ij}^k, \quad \forall i=2,3,...,N,k=1,2,...,K;\\
& \sum\limits_{j=2}^N x_{1j}^k=1-x_{1,N+1}^k, \quad \forall k=1,2,...,K;\\
& \sum\limits_{i=2}^N x_{i,N+1}=1-x_{1,N+1}^k, \quad \forall k=1,2,...,K;\\
& \sum_{j\neq i,j=2,3,...,N} x_{ij}^k\leq 1-x_{1,N+1}^k, \quad \forall i=2,3,...,N, k=1,2,...,K;\\
& \sum_{j\neq i}x_{ij}^k=\sum_{j\neq i}x_{ji}^k, \quad \forall i=2,3,...,N, k=1,2,...,K;\\
& W_1^k=w\sum\limits_{i=2}^N Q_{i1}^k+W_0^k, \quad \forall k=1,2,...,K;\\
& W_j^k\leq W_i^k+w(Q_{j2}^k-Q_{j1}^k)+M(1-x_{ij}^k), \quad \forall i=1,2,...,N, j=2,3,...,N+1, i\neq j, k=1,2,...,K;\\
& W_j^k\geq W_i^k+w(Q_{j2}^k-Q_{j1}^k)-M(1-x_{ij}^k), \quad \forall i=1,2,...,N, j=2,3,...,N+1, i\neq j, k=1,2,...,K;\\
& W_i^k\leq W_{\max}^k, \quad \forall i=1,2,...,N, k=1,2,...,K;\\
& I_1^k=I_{N+1}^k, \quad \forall k=1,2,...,K;\\
& I_j^k\leq I_i^k+B_i^k-d_{ij}g_k(W_i^k,t_{ij}^k)+M(1-x_{ij}^k), \quad \forall i=1,2,...,N, j=2,3,...,N+1, i\neq j, k=1,2,...,K;\\
& I_j^k\geq I_i^k+B_i^k-d_{ij}g_k(W_i^k,t_{ij}^k)-M(1-x_{ij}^k), \quad \forall i=1,2,...,N, j=2,3,...,N+1, i\neq j, k=1,2,...,K;\\
& y_i^k\leq \sum_{j\neq i}x_{ij}^k, \quad \forall i=1,2,...,N, k=1,2,...,K;\\
& e^kI_m^ky_i^k\leq B_i^k\leq I_m^ky_i^k, \quad \forall i=1,2,...,N, k=1,2,...,K;\\
& I_i^k\geq c^kI_m^k, \quad \forall i=1,2,...,N+1, k=1,2,...,K;\\
& I_i^k+B_i^k\leq I_{\max}^k, \quad \forall i=1,2,...,N, k=1,2,...,K;\\
& a_1^k = 0, \quad \forall k=1,2,...,K;\\
& a_j^k\geq a_i^k+d_{ij}t_{ij}^k+s_i^k-M(1-x_{ij}^k), \quad \forall i=1,2,...,N, j=2,3,...,N+1, i\neq j, k=1,2,...,K;\\
& a_j^k\leq a_i^k+d_{ij}t_{ij}^k+s_i^k+M(1-x_{ij}^k), \quad \forall i=1,2,...,N, j=2,3,...,N+1, i\neq j, k=1,2,...,K;\\
& e_i\leq a_i^k\leq l_i, \quad \forall i=2,3,...,N, k=1,2,...,K;\\
& a_{N+1}^k\leq T^k, \quad \forall k=1,2,...,K;\\
& Q_{i1}^k\geq 0, Q_{i2}^k\geq 0, \quad \forall i=2,3,...,N, k=1,2,...,K;\\
& x_{ij}^k=0 \text{ or } 1, \quad \forall i=1,2,...,N, j=2,3,...,N+1, i\neq j, k=1,2,...,K;\\
& y_i^k=0 \text{ or } 1, \quad \forall i=1,2,...,N, k=1,2,...,K.
\end{align}
\setcounter{equation}{\value{mytempeqncnt}}
\hrulefill
\vspace*{4pt}
\end{figure*}

The objective function is to maximize the total profit for the fleet, where the first term is the revenue earned by fulfilling demands from customer ports, the second term is the revenue earned by chartering out ships, and the last term is the total fuel costs, where $f_i(B_i^k)$ is the fuel cost function considering increment quantity price discounts at port $i$ and defined as
\[\left\{ {\begin{array}{*{20}{c}}
   {p_i^1 B_i^k}  \\
   {p_i^1 v_1+p_i^2(B_i^k-v_1)}  \\
   {p_i^1 v_1+p_i^2(v_2-v_1)+p_i^3(B_i^k-v_2)}  \\
\end{array}{\rm{  }}\begin{array}{*{20}{c}}
   {B_i^k\in[0,v_1]}  \\
   {B_i^k\in[v_1,v_2]}  \\
   {B_i^k\in[v_2,v_3]}  \\
\end{array}} \right.
\]

Constraint (1) says that the quantity served at each customer port should not exceed its demand. Constraint (2) states that if a customer port is not visited by a ship, there is no demand that can be served. Constraints (3) to (6) define the routings for the fleet. (3) and (4) ensures that ships start and end their journey at the depot if it is used by the fleet. (5) makes sure that no internal voyages is traveled by ships chartered out. And (6) is the flow conservation constraint. (7) to (10) are constraints subject to displacements of ships. (7) is the initial status, while (8) and (9) describe the changes in the displacement of ships in terms of the quantity of cargos they are carrying. (10) indicates the deadweight displacement for ships. It is worth to mention here that these constraints with deadweight defined for each ship also serve as the capacity constraints in terms of the largest amount of cargo a ship can carry. Constraints (11) to (17) describe the amount of fuel carried by ships along their journey. Constraint (11) assumes that the amount of fuel a ship carries at the beginning and the end of the journey are the same. (12) and (13) together describe the dynamic of fuel consumption along each traveled voyage, where $g_k(W_i^k,t_{ij}^k)$ is the fuel consumption of ship $k$ per unit distance and defined as $g_k(W_i^k,t_{ij}^k)=C^k\frac{W_i^k}{(t_{ij}^k)^2}$, where $C^k$ is a specific constant for ship $k$. (14) is the indicator constraint to make sure that a ship can only bunker at some port it visits. (15) to (17) state the upper limits and lower limits of bunkering amount and fuel levels. Constraints (18) to (22) are the time window constraints. (19) and (20) together describe the time spent on each voyage, and (21) and (22) states the time window at each port. Note that with time window constraints (19) and (20), subtour elimination constraints are no longer necessary to define the problem. At last, constraint (23) to (25) define the decision variables.

\section{Heuristics}
The problem is NP-hard since it is a generalization of the classical VRP. When there are $N$ ports (including depot) and $K$ ships in the problem, there are $N(N-1)$ voyages to be considered. The number of variables is $15KN^2+5KN$, in which $9KN^2+K$ of them are binary variables. The complexity determines the rapid growth of computational time. Only small problems can be solved by exact algorithms using the formulated MIP. Hence, heuristic algorithms are necessary.

The decision variables are tightly coupled. The routing decision determines the set of visiting ports and fuel consumption for each ship, and hence the bunkering ports and bunkering amounts depend on routing and scheduling. On the other hand, the routing has to be adjusted in order to reduce total fuel consumption and bunkering costs. It is desirable for a heuristic to keep the advantage of the joint optimization of these decisions.

Note that if the sets of visiting ports for each ship are known, the decisions can be made independently for each single ship. The original problem can then be partially decomposed into $K$ single-vehicle problems. We can still use the formulation to solve each of them, since it contains only one ship and potentially much less number of ports. Even for large-scale problems, the single-vehicle problems can be still efficiently solved using optimization tools. On the other hand, the coupling of decision variables within each ship is kept. Hence, we propose a two-stage heuristic. In the first stage, the assignment of ports to ships is generated, and in the second stage decisions are made using optimization tools for each single ship based on the assignment. The two stages are iteratively carried out to find the best assignment. The main challenge is then to find an efficient way of searching a promising assignment.

The heuristic is also reasonable from a managerial perspective. Determining the assignment can be seen as a tactic level decision, while routing, scheduling and profit management can be treated as operational level decisions.

We first define a set of binary vectors $y_k=[y_k(1),y_k(2),...,y_k(N-1)],k=1,2,...,K$, where $y_k(i)=1$ indicates that customer port $i$ is served ship $k$, and 0 otherwise. $Y\in \Theta$ is then defined based on $y_k$: $Y=[y_1,y_2,...,y_K]$. $Y$ can be thought as a variable which assigns customer ports to ships. Given a specific assignment, a simple optimization problem can be solved for each single ship. Let $h(Y)$ be the total optimized profit when customer ports are assigned to ships as $Y$. The key problem is then to find an optimized assignment $Y$ for ports.

\subsection{Single ship problems}
Given an assignment $Y$, the optimal decisions can be made for each single ship independently. The following formulation can efficiently solve each single ship problem with the number of visiting ports as $n$.
\newcounter{tempeqncnt}
\begin{figure*}
\normalsize
\setcounter{tempeqncnt}{\value{equation}}
\setcounter{equation}{25}
\begin{align}
\max\quad & \sum\limits_{i=2}^n (r_{i1}Q_{i1}+r_{i2}Q_{i2})-\sum\limits_{i=1}^n f_i(B_i)\nonumber\\
\mathrm{s.t.}\quad & Q_{i1}\leq D_{i1}, Q_{i2}\leq D_{i2}, \forall i=2,3,...,n;\nonumber\\
& \sum_{j\neq i}x_{ij}=1, \forall i=1,2,...,n;\\
& \sum_{j\neq i}x_{ji}=1, \forall i=2,3,...,n+1;\\
& W_1=w\sum\limits_{i=2}^n Q_{i1}+W_0;\nonumber\\
& W_j\leq W_i+w(Q_{j2}-Q_{j1})+M(1-x_{ij}), \forall i=1,2,...,n, j=2,3,...,n+1, i\neq j;\nonumber\\
& W_j\geq W_i+w(Q_{j2}-Q_{j1})-M(1-x_{ij}), \forall i=1,2,...,n, j=2,3,...,n+1, i\neq j;\nonumber\\
& W_i\leq W_{\max}, \forall i=1,2,...,n;\nonumber\\
& I_1=I_{n+1};\nonumber\\
& I_j\leq I_i+B_i-d_{ij}g(W_i,t_{ij})+M(1-x_{ij}), \forall i=1,2,...,n, j=2,3,...,n+1, i\neq j;\nonumber\\
& I_j\geq I_i+B_i-d_{ij}g(W_i,t_{ij})-M(1-x_{ij}), \forall i=1,2,...,n, j=2,3,...,n+1, i\neq j;\nonumber\\
& eI_my_i\leq B_i\leq I_my_i, \forall i=1,2,...,n;\nonumber\\
& I_i\geq cI_m, \forall i=1,2,...,n+1;\nonumber\\
& I_i+B_i\leq I_{\max}, \forall i=1,2,...,n;\nonumber\\
& a_1 = 0;\nonumber\\
& a_j\geq a_i+d_{ij}t_{ij}+s_i-M(1-x_{ij}), \forall i=1,...,n, j=2,...,n+1, i\neq j;\nonumber\\
& a_j\leq a_i+d_{ij}t_{ij}+s_i+M(1-x_{ij}), \forall i=1,...,n, j=2,...,n+1, i\neq j;\nonumber\\
& e_i\leq a_i\leq l_i, \forall i=2,3,...,n;\nonumber\\
& a_{n+1}\leq T;\nonumber\\
& Q_{i1}\geq 0, Q_{i2}\geq 0, \forall i=2,3,...,n;\nonumber\\
& x_{ij}=0 \text{ or } 1, \forall i=1,2,...,n, j=2,3,...,n+1, i\neq j;\nonumber\\
& y_i=0 \text{ or } 1, \forall i=1,2,...,n.\nonumber
\end{align}
\setcounter{equation}{\value{tempeqncnt}}
\hrulefill
\vspace*{4pt}
\end{figure*}

Note that constraints (26) and (27) ensures that the ship visits every customer ports that are assigned to it.

Using this formulation, we can calculate the optimal profit with any given assignment $Y$ as in Table \ref{tabSingle}. Note that in Step 4, we will examine whether the profit by serving customer ports is larger than the profit of chartering out the ship. If not, the ship will not set out.

\begin{table}[ptbh]
\caption{Calculating optimal profit for a given assignment $Y$}%
\label{tabSingle}
\renewcommand{\arraystretch}{1.5}
\centering
\par%
\begin{tabular}
[l]{p{28pt}p{190pt}}\hline\hline \textbf{Step 1}: &$k=1$, set the demands at each customer ports as the real demands;\\
\textbf{Step 2}: & For ship $k$, check the $(k-1)*(N-1)+1$ to $k*(N-1)$ digits of $Y$, and determine the visiting ports;\\
\textbf{Step 3}: & Solve a single ship problem for the visiting ports, with optimal profit $p_k^*$ and pickup and delivery amount $Q_{i1}^*$, $Q_{i2}^*$;\\
\textbf{Step 4}: & If the single ship profit exceeds the chartering profit for the ship, delete the fulfilled amount from the ports' demands. Otherwise, set this ship as chartered out, and $p_k^*=R_k$. If $k<K$, set $k=k+1$, and go to step 2;\\
\textbf{Step 5}: & Calculate the optimal profit $h(Y)=\sum\limits_{k=1}^K p_k^*$.\\
\hline\hline
\end{tabular}
\end{table}

\subsection{Neighborhood Search}
There are some existing heuristics to search for a local optimum within discrete variables. The most common one is neighborhood search. For chosen assignment, its neighborhood is examined. If a better assignment is found, that assignment is chosen and the procedure is repeated.
In this paper, we will continue to use the definition of neighborhood in the last section. The algorithm is shown in Table \ref{tabNS}.

\begin{table}[ptbh]
\caption{Heuristic based on neighborhood search}%
\label{tabNS}
\renewcommand{\arraystretch}{1.5}
\centering
\par%
\begin{tabular}
[l]{p{28pt}p{190pt}}\hline\hline \textbf{Step 1}: & Generate an initial solution $Y_0\in \Theta$, calculate $h(Y_0)$. Set $Y^*=\{Y_0\}$ and $h^*=h(Y_0)$;\\
\textbf{Step 2}: & Examine neighborhood of $Y^*$ by changing its digits once at a time, if there is $Y_1\in N(Y^*)$ such that $h(Y_1)>h(Y^*)$, set $Y^*=\{Y_1\}$ and $h^*=f(Y_1)$, go to Step 2;\\
\textbf{Step 3}: & $Y^*$ is a local optimal solution.\\
\hline\hline
\end{tabular}
\end{table}

However, the size of neighborhood increases with the increment of number of ports and ships. A more efficient way of converging is still needed for large problems.

\subsection{Heuristic based on Most Promising Area}
Hong and Nelson \cite{HongNelson06} proposed an algorithm termed convergent optimization via most-promising-area stochastic search (COMPASS), based on a unique neighborhood structure. The neighborhood is defined in each iteration, and is fully adaptive. The idea that the most promising area is defined as the set of feasible solutions that are at least as close to the current best solution as they are to other visited ones, can be utilized to solve the fleet problem in this paper. Our algorithm is going to modify COMPASS in three aspects to solve our problem. First, COMPASS solves problems with integer solutions, while this problem only deals with binary variables. This also results in the second difference, that we will not use Euclidean distance as the definition of distances between points, which is used in COMPASS. The last difference is that COMPASS is proposed for stochastic problems, but the ways of searching neighborhood is also suitable for the deterministic problem in this paper.

The distance between two assignment variable $Y_1$ and $Y_2$, $dist(Y_1,Y_2)$ is defined as the number of different digits, i.e. the number of $i$'s such that $Y_1(i)\neq Y_2(i)$. Let $N(Y)=\{Z:Z\in \Theta \text{ and } dist(Y,Z)=1\}$ be the local neighborhood of $Y\in \Theta$. $Y$ is a local optimum if $h(Y)\geq h(Z)$ for all $Z\in N(Y)$. Let $\Phi$ denote the set of local optima in $\Theta$.

In this algorithm, we use $R_r$ to denote the set of all solutions visited through iteration $r$, and use $Y_r^*$ to denote the solution with the largest total profit among all $Y\in R_r$. At the end of iteration $r$, we construct the most promising area $\zeta_r=\{Y:Y\in \Theta \text{ and } dist(Y,Y_r^*)\leq dist(Y,Z), \forall Z\in R_r \text{ and } Z\neq Y_r^*\}$. The set $\zeta_r$ includes all feasible solutions that are at least as close to $Y_r^*$ as to other solutions in $R_r$. At iteration $r+1$, $m$ solutions will be randomly generated from $\zeta_r$ uniformly and independently. The algorithm based on Most Promising Area is shown in Table \ref{tabMPA}.

\begin{table}[ptbh]
\caption{Heuristic based on Most Promising Area}%
\label{tabMPA}
\renewcommand{\arraystretch}{1.5}
\centering
\par%
\begin{tabular}
[l]{p{28pt}p{190pt}}\hline\hline \textbf{Step 1}: & Set iteration count $r=0$. Generate an initial solution $Y_0\in \Theta$, set $R_0=\{Y_0\}$ and $Y_0^*=Y_0$. Calculate $h(Y_0)$. $\zeta_0=\Theta$;\\
\textbf{Step 2}: & Let $r=r+1$. Generate $Y_{r1},Y_{r2},...,Y_{rm}$ from $\zeta_{r-1}$ uniformly and independently, and calculate $h(Y_{ri})$ for every newly generated assignment;\\
\textbf{Step 3}: & Let $R_r=R_{r-1}\cup \{Y_{r1},Y_{r2},...,Y_{rm}\}$, and $Y_r^*=\operatorname{arg\,max}_{Y\in R_r} h(Y)$. Construct $\zeta_r$. If $Y_r^*$ has not changed for certain iterations, stop. Else, and go to step 2.\\
\hline\hline
\end{tabular}
\end{table}

Other terminal conditions can be adopted. More iterations can be examined for precision. Or the algorithm can be stopped whenever all of the computational budget is consumed.

It is worthwhile to note that this algorithm is efficient even for large scale problems. The number of points generated in each iteration is fixed, and will not affected by the size of neighborhood. The computational time is hence proportional to the number of iterations, number of points generated in each iteration, the number of ships, and the time used for each single ship problem. The computational complexity is then mainly dependent on that of single ship problems. Due to time and capacity constraints, the number of ports assigned to a ship is limited, so each single ship problem can be solved efficiently.

Of course we have to make sure the convergence of the algorithm. The proof is similar to the one in \cite{HongNelson06}, but we need to prove it under our definitions.

\begin{theorem} The infinite sequence $\{Y_0^*,Y_1^*,...\}$ generated by the algorithm converges with probability 1 to the set $\Phi$, i.e. $P\{Y_r^*\in \Phi\quad i.o.\}=1$.
\end{theorem}

\begin{proof}
For any infinite sequence $\{Y_0^*,Y_1^*,...\}$ generated by the algorithm, $A=R_\infty$ exists since $R\in \Theta$ and $\Theta$ is a finite set.
Since $|R_\infty|<\infty$, there exists a $\kappa >0$ such that $R_r=R_\infty=A$ for all $r\geq \kappa$. Then for $r\geq \kappa$,
\[h(Y_r^*)=\max_{Y\in R_r} h(Y)=\max_{Y\in R_\infty} h(Y)\]
Hence
\setcounter{equation}{27}
\begin{equation}\label{eqt1}
P\{\lim_{r\to \infty} h(Y_r^*)=\max_{Y\in R_\infty} h(Y)|R_\infty=A\}=1
\end{equation}
Note that
\begin{eqnarray}
P\{Y_r^*\in \Phi\quad i.o.\}=\qquad\qquad\qquad\qquad\qquad\qquad\qquad\nonumber\\
\quad\quad\quad\sum_A P\{Y_r^*\in \Phi\quad i.o.|R_\infty=A\}P\{R_\infty=A\}\nonumber
\end{eqnarray}

Then proving the theorem is equivalent to proving $P\{Y_r^*\in \Phi\quad i.o.|R_\infty=A\}=1$ for any nonempty $A$ such that $P\{R_\infty=A\}>0$. Since $Y_r^*=\operatorname{arg\,max}_{Y\in R_r} h(Y)$, there exists a $Y$ such that $Y\in A, Y\not\in \Phi$ and $Y_r^*=Y$. If $Y_r^* \not\in \Phi\quad i.o.,$ it is not a local optimum, and there exists a $Z\in N(Y)$ and $h(Z)>h(Y)$. For any other solution $Y'\in R_r$, $dist(Z,Y)=1\leq dist(Z,Y')$. According to the definition of most promising area, $Z\in \zeta_r$. As solutions are generated uniformly in $\zeta_r$, we have
\[P\{Z\in R_{r+1}|Y_r^*=Y \text{ and } Z\not\in R_r\}=\frac{1}{|\zeta_r|}\geq \frac{1}{|\Theta|}>0\]
Hence, $P\{Z\in A|R_\infty=A,Y_r^*=Y\}=1$. Note that if $Z\in A$, then $Y$ is not a local optimum since $h(Z)>h(Y)$. By equation (\ref{eqt1}), we know that with probability 1, $Y_r^*$ can only equal to $Y$ finitely many times, which contradicts with "infinitely often". Therefore, $Y_r^*\in \Phi\quad i.o.$ with probability 1, which proves the theorem.
\end{proof}

\subsection{Improvements}
As the generation of assignments is actually a random generation of a vector of binary variables, it is highly possible that there are few or even none remaining demands for a port after visited by other ships. Although there are enough demands, the ship may not doing any services for the port due to capacity constraints. In these two situations, the ship will still visit the ports although there is no service activity there. We will show that it is not desirable for a ship to visit a port if no service and no bunkering happens there, under the assumption that the travelling distance between any two ports $i,k$ is no larger than that the distance with another port $j$ visited between them, i.e. $d_{ik}\leq d_{ij}+d_{jk}$.

We will first prove a lemma that same velocity should be used on the voyages to and from the no-service port.

\begin{lemma}\label{lemma1} If a ship visit port $i,j,k$ in sequence, while there is no pickup, delivery and bunkering activities at port $j$, then the cruise speed is equal on two voyages, i.e. $v_{ij}=v_{jk}$.
\end{lemma}

\begin{proof}
If $v_{ij}\neq v_{jk}$, WLOG, assume $v_{ij}<v_{jk}$. Since the function $f(v)=(d_{ij}+d{jk})/v$ is continuous and monotonically decreasing on its domain, and
\[\frac{d_{ij}+d{jk}}{v_{jk}}<\frac{d_{ij}}{v_{ij}}+\frac{d_{jk}}{v_{jk}}<\frac{d_{ij}+d{jk}}{v_{ij}}\]
Hence, there exists a $v^*\in [v_{ij},v{jk}]$ such that
\[\frac{d_{ij}+d{jk}}{v^*}=\frac{d_{ij}}{v_{ij}}+\frac{d_{jk}}{v_{jk}}\].
The fuel consumption on the two voyages is $CW_{ij}d_{ij}v_{ij}^2+CW_{jk}d_{jk}v_{jk}^2$, with $C$ as a constant, and $W_{ij}$ as the weight of ship on voyage $ij$. Since there is no service and no bunkering happening at port $j$, the weight of ship will not change, $W_{ij}=W_{jk}=W$. Hence the fuel consumption is $CW(d_{ij}v_{ij}^2+d{jk}v_{jk}^2)$.
On the other hand, if the ship travel through the two voyage with speed $v^*$, the fuel consumption is $CW(d_{ij}+d_{jk})v^{*2}$. We want to show that the amount of fuel consumption is less with speed $v^*$, which is equal to show that
\begin{equation}\label{eqt2}
d_{ij}v_{ij}^2+d{jk}v_{jk}^2>(d_{ij}+d_{jk})v^{*2}
\end{equation}
With
\[v^*=\frac{(d_{ij}+d_{jk})v_{ij}v_{jk}}{d_{ij}v_{jk}+d_{jk}v_{ij}}\]
The inequality (\ref{eqt2}) can be derived as
\begin{eqnarray}
d_{ij}(2v_{ij}^3v_{jk}+v_{ij}^4)+d_{jk}(v_{ij}^4+2v_{ij}v_{jk}^3)\quad\quad\quad\quad\quad\nonumber\\
\quad\quad\quad>3d_{ij}v_{ij}^2v_{jk}^2+3d_{jk}v_{ij}^2v_{jk}^2\label{eqt3}
\end{eqnarray}

Since $2v_{ij}^3v_{jk}+v_{jk}^4-3v_{ij}^2v_{jk}^2=v_{jk}(v_{ij}-v_{jk})^2(2v_{ij}+v{jk})>0$, $d_{ij}(2v_{ij}^3v_{jk}+v_{ij}^4)>3d_{ij}v_{ij}^2v_{jk}^2$. Similarly, $d_{jk}(v_{ij}^4+2v_{ij}v_{jk}^3)>3d_{jk}v_{ij}^2v_{jk}^2$. Hence, the inequality (\ref{eqt3}) holds, and (\ref{eqt2}) also holds, which indicates that fuel consumption is less using one speed. Since the total travel time on the two voyages is the same according to the definition of $v^*$, $v^*$ is a feasible and better solution.
\end{proof}

Using this fact, we will show that any port without service and bunkering should not be included in voyages for maximum profit.

\begin{proposition} Suppose $d_{ij}\leq d_{ik}+d_{kj}$ holds for any set of ports. A ship visit port $j$, while there is no pickup, delivery and bunkering activities, then port $j$ should not be assigned to the ship for maximum profit.
\end{proposition}

\begin{proof}
If port $j$ is visited, the voyages from $i$ to $j$, and from $j$ to $k$ is contained as a part of the optimal routing. According to Lemma \ref{lemma1}, the cruise speed on the two voyages are the same, i.e. $v_{ij}=v_{jk}=v$. We can always construct a new routing for the ship with port $i$ to port $k$ directly using speed $v$, and the other part of the cycle remains the same. With the same speed and less traveling distance, traveling time is reduced, and the deadline will not be violated, which indicates that the new routing is a feasible one.

Since no service is done, the weight of ship is not changed. With the same weight and velocity, the amount of fuel consumption is proportional to the traveling distance. Since $d_{ik}\leq d_{ij}+d_{jk}$, the fuel consumption is less in the new routing. Since visiting port $j$ does not produce revenues, the new routing will result in more profit.
\end{proof}

When we get a solution of a single ship problem for a given assignment, we can always examine whether there are ports with no service and bunkering. If such ports exist, we should delete these ports from the assignment, and re-optimize the routing for that ship. The two-stage heuristic algorithm can be improved to a three-stage algorithm, with the third step of refining solutions. This step gives an opportunity to find a better solution for assignment in each iteration, and help faster convergence to local optimum.

\section{Numerical Results}
\subsection{Effects of weight and bunkering}
In this section, a simple example will be given, in which 7 ports are to be considered. The proposed model is a very complex formulation since several factors are comprehensively considered in the optimization. The example will be used to illustrate the necessity of considering these factors by examining the effect on the optimal routing if a factor is not included. The experiment is carried out based on ports information listed in Table \ref{tabPortInfo}. The exact optimal profit obtained by the model is $1.6356\times10^6$ with routing shown in Figure \ref{BW}.

\begin{table}[h]
\caption{Ports information}%
\label{tabPortInfo}
\renewcommand{\arraystretch}{1.5}
\centering
\par%
\begin{tabular}{|p{10pt}|p{30pt}|p{26pt}|p{26pt}|p{26pt}|p{25pt}|p{18pt}|p{18pt}|}
\hline Port & Location (*100) & Delivery demand & Unit delivery revenue & Pickup demand & Unit pickup revenue & Dead line & Fuel price\\
\hline
\textbf{1} & (0,0) & N.A. & N.A. & N.A. & N.A. & 168 & 677.5\\
\textbf{2} & (-8,2) & 0 & 130 & 3600 & 127 & 120 & 629\\
\textbf{3} & (0,10) & 3000 & 158 & 2700 & 133 & 90 & 673.5\\
\textbf{4} & (-4,-6) & 2500 & 105 & 2500 & 160 & 80 & 679.5\\
\textbf{5} & (4.7,-3) & 2700 & 120 & 0 & 95 & 40 & 630\\
\textbf{6} & (-2,9) & 3200 & 132 & 0 & 120 & 120 & 635\\
\textbf{7} & (-4,2) & 0 & 90 & 3500 & 133 & 120 & 655.5\\
\hline
\end{tabular}
\end{table}

\begin{figure}
\centerline{\includegraphics[width=0.4\textwidth]{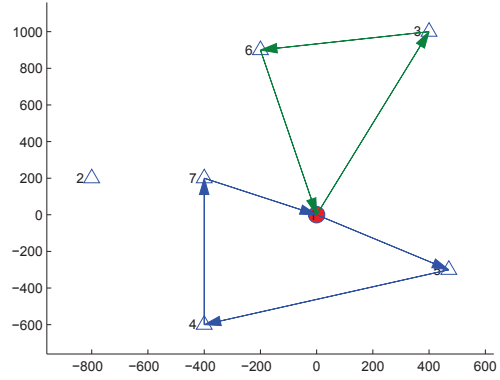}}\caption{True optimal routing}%
\label{BW}%
\end{figure}

In the first variation case, the change of weight of ships is ignored and assumed to be constant throughout their journey. Hence, the fuel consumption rate is only related to cruise speed and the effect of weight is not considered, the optimal routing is then as shown in Figure \ref{BnW}. Port 2 is serviced instead of port 5 due to high demand and relatively high revenue of port 2. However, the demand from port 2 is all pickup cargos, the higher demand indicates a voyage with large carrying weight. Furthermore, the distance between port 2 and the depot is relatively far, hence, in Figure \ref{BnW}, the increase in revenue will be largely compromised by the increase in fuel consumption because weight contributes to fuel consumption rate. The optimal solutions without considering ship weight change will give a profit of $1.486\times10^6$, which is 9.15\% lower than the optimal one in Figure \ref{BW}. Based on 100 random sample cases, ignoring the change of ship weight  results in an average decrease of 8.71\% in profit.

On the other hand, if information of fuel prices at different ports are not utilized, the model might also result in inferior solutions. In this example, if the different fuel prices are neglected, and assume that the fuel prices at all ports are set at the average level, the optimal routing will be as shown in Figure \ref{nBW}. Port 5 has a competitive fuel price but a lower revenue compared with nearby ports. Hence, visiting port 5 provides a benefit in the fuel cost. While all ports share the same fuel price, the port is not as profitable as other ports. The optimal solutions without considering different fuel prices will give a profit of $1.600\times10^6$, which is 2.17\% lower than the optimal one in Figure \ref{BW}. Based on 100 random sample cases, neglecting different fuel prices at ports results in an average decrease of 5.74\% in profit.

\begin{figure}
  \centering
  \subfigure[When weight is not considered]{
    \label{BnW} 
    \includegraphics[width=0.4\textwidth]{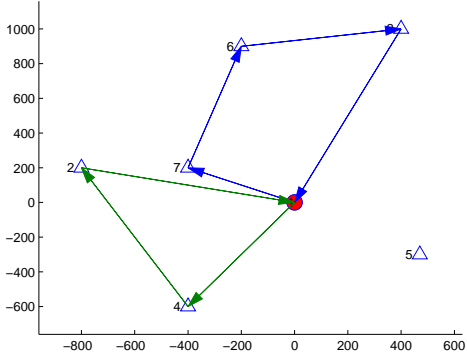}}
  \hspace{0.4in}
  \subfigure[When bunkering is not considered]{
    \label{nBW} 
    \includegraphics[width=0.4\textwidth]{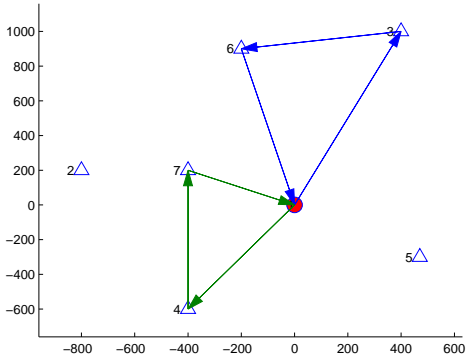}}
  \caption{Optimal routing under variate conditions}
  \label{variations} 
\end{figure}

\subsection{Computational time}
The formulation can help fleet managers make decision is aspects including choosing serving ports, routing and scheduling, as well as refueling policy. Such a comprehensive formulation leads to rapidly increasing number of variables and constraints.

The proposed two-stage heuristic, on the other hand, performs better in terms of computational time, especially when $N$ is large. Although the subproblems are also solved using CPLEX in the heuristic, a single ship problem is solved each time, which reduces the complexity to a large extent. As a matter of fact, using CPLEX may consume up the memory of computer when $N$ is large, while the heuristic algorithm can still provide answers when it occurs.

The experiment is carried out such that ports information is generated randomly, and three ways are used to solve the problem. The experiment are all run using Matlab2010b in Windows 7, on PC with CPU of i5-2400@3.1GHz, and RAM of 4GB. Since CPLEX in Matlab can only guarantee to successfully solve the problems for $N<10$, samples are generated for $N=7,8,9$, and $K=2$. Table \ref{tabCompTime} shows the average computational times used by (1) CPLEX (exact optimal solution), (2) MPAS and (3) heuristic based on neighborhood search (later referred to as NS) respectively. Table \ref{tabGap} shows the performances of the two heuristics by the optimality gaps.With the increasing of the number of ports, compared with the rapid increment in computational time to get exact optimal solutions, the growth of time for heuristics is not too steep, with reasonably close optimal solutions. Table \ref{tabGap} may also suggest that the optimality gaps tend to decrease with the increasing number of ports, which makes the heuristics more desirable. The performance in terms of computational time can be further improved if it is programmed in C++.

\begin{table}
\caption{Computational times}%
\label{tabCompTime}
\renewcommand{\arraystretch}{1.5}
\centering
\par%
\begin{tabular}{|p{18pt}|p{40pt}|p{40pt}|p{40pt}|}
\hline N & CPLEX & MPA & NS\\
\hline
7 & 37.88 & 30.02 & 40.13\\
8 & 63.78 & 39.32 & 59.23\\
9 & 1053.07 & 38.37 & 64.27\\
\hline
\end{tabular}
\end{table}

\begin{table}
\caption{Optimality gap}%
\label{tabGap}
\renewcommand{\arraystretch}{1.5}
\centering
\par%
\begin{tabular}{|p{10pt}|p{35pt}|p{35pt}|p{35pt}|p{35pt}|}
\hline N & MPA& MPA & NS  & NS \\
 &  (average) &  (max) &  (average) &  (max)\\
\hline
7 & 2.10\% & 16.70\% & 2.40\% & 16.70\%\\
8 & 0.14\% & 2.65\% & 1.08\% & 18.80\%\\
9 & 0.23\% & 0.31\% & 0.75\% & 0.33\%\\
\hline
\end{tabular}
\end{table}

Figure \ref{CompTime} shows the average computational times used by the two heuristics with respect to number of ports. With the growth of the problem size, the neighborhood of each solution is also increasing. Since the heuristic based on NS need to examine neighborhoods to find local optimum, the computational time tends to increase much faster with the growth of port number, although it performed well when facing small-size problems. The heuristic based on MPA, on the other hand, did not show a obvious rising speed in growth of computational time, and is more advantageous in terms of efficiency.

\begin{figure}
  \centering
  \subfigure[Computational times for MPA]{
    \label{MPAtime} 
    \includegraphics[width=0.4\textwidth]{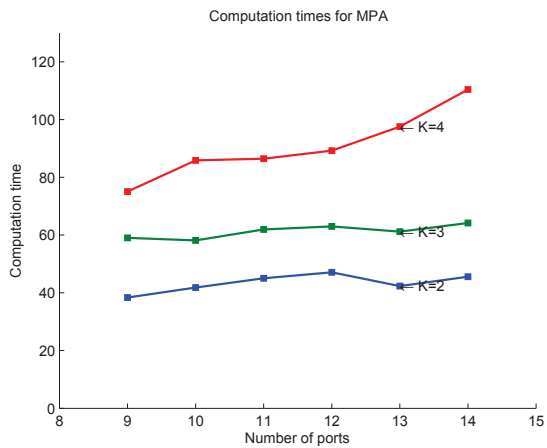}}
  \hspace{0 in}
  \subfigure[Computational times for NS]{
    \label{NStime} 
    \includegraphics[width=0.4\textwidth]{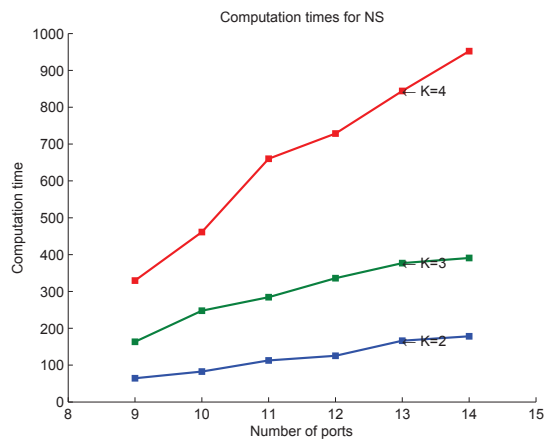}}
  \caption{Comptational times for two heuristics}
  \label{CompTime} 
\end{figure}

\subsection{Case study}
We now apply our model and solution approach to a shipping company with the depot settled in Hongkong. The company provides a set of shipping services between some ports in Asia (such as Singapore, Hong Kong, Ningbo, Qingdao, Shanghai, Xingang, Yantian, Xiamen, Kaosiung, Chiwan and Yokohama). In order to decide the routing and bunkering strategies, we collect the relevant parameters including distances between ports, demand, revenues, and fuel prices. The heuristic algorithm is then carried out to solve the problem. 14 ports and 5 ships are included in the problem, and some parameters are shown in Table \ref{tabparameter}.

\begin{table}
\caption{Parameters for case study}%
\label{tabparameter}
\renewcommand{\arraystretch}{1.5}
\centering
\par%
\begin{tabular}{p{120pt}p{70pt}}
\hline Parameter & Value\\
\hline
Number of ports of calls & 13\\
Number of ships & 5\\
Ship size & 5000 TEU\\
Ship bunker fuel capacity & 1500 ton\\
Ship speed  interval & (14,24) knots\\
Bunker fuel price discount & 10\%,20\%\\
Bunker fuel consumption parameter & $C_k=7.55\times10^{-7}$\\
\hline
\end{tabular}
\end{table}

Note that in reality, it is usually not desirable to assign many ports to one ship, due to the constraints of capacities and cycle times. Hence, when solving these real problems with our heuristic algorithms, only small sized single-ship problems are to be solved in each iteration. With 7 or 8 customer ports, a single-ship problem can be solved within seconds. Hence, our heuristic algorithm shows great efficiency when solving real problems. For the case shown, it took less than half an hour to give the final solutions.

The optimal solution for each ship is shown in Table \ref{caseResult1} - \ref{caseResult5} respectively. From the results of the study, we can draw some insights. Ships tend to do delivering services first followed by pickup activities, even when sometimes it causes geographically more traveling distances. Apart from the capacity concern, the savings on fuel consumption due to less weight of ships may outperform the cost of a longer voyage.

Although 13 ports of calls are considered in the problem, only 11 ports are actually visited if the proposed solution is implemented. Two ports are excluded from the solution due to small demands.

\begin{table}
\caption{Optimal routing and fuel strategy: Ship 1}%
\label{caseResult1}
\renewcommand{\arraystretch}{1.5}
\centering
\par%
\begin{tabular}{p{34pt}p{25pt}p{25pt}p{25pt}p{25pt}p{30pt}}
\hline Port & Delivery (TEU) & Pickup (TEU) & Velocity (knots) & Arrival time & Bunkering Amount\\
\hline
Hongkong & & & & 0 & 673\\
Yokohama & 3398 & 2755 & 14.0 & 115 & \\
Yantian & 644 & 2245 & 14.0 & 230 &\\
Hongkong & & & 14.0 & 240 & \\
\hline
\end{tabular}
\end{table}

\begin{table}
\caption{Optimal routing and fuel strategy: Ship 2}%
\renewcommand{\arraystretch}{1.5}
\centering
\par%
\begin{tabular}{p{34pt}p{25pt}p{25pt}p{25pt}p{25pt}p{30pt}}
\hline Port & Delivery (TEU) & Pickup (TEU) & Velocity (knots) & Arrival time & Bunkering Amount\\
\hline
Hongkong & & & & 0 & \\
Singapore & 4210 & & 14.6 & 100 &\\
Qingdao & 537 & & 19.6 & 224 & \\
Dalian & 213 & & 21.4 & 237 & \\
Xingang & 40 & 437 & 22.2 & 247 & \\
Shanghai & & 4204 & 24.0 & 279 & \\
Ningbo & & 359 & 17.0 & 286 & 1400\\
Hongkong & & & 14.0 & 308 & \\
\hline
\end{tabular}
\end{table}

\begin{table}
\caption{Optimal routing and fuel strategy: Ship 3}%
\renewcommand{\arraystretch}{1.5}
\centering
\par%
\begin{tabular}{p{34pt}p{25pt}p{25pt}p{25pt}p{25pt}p{30pt}}
\hline Port & Delivery (TEU) & Pickup (TEU) & Velocity (knots) & Arrival time & Bunkering Amount\\
\hline
Hongkong & & & & 0 & \\
Kaosiung & 4288 &  & 14.0 & 25 & 1011\\
Busan & 631 & & 17.0 & 79 & \\
Xiamen & 81 & 2092 & 17.0 & 132 & \\
Singapore & & 1469 & 16.7 & 230 & \\
Chiwan & & 1439 & 14.0 & 333 & \\
Hongkong & & & 14.0 & 335 & \\
\hline
\end{tabular}
\end{table}

\begin{table}
\caption{Optimal routing and fuel strategy: Ship 4}%
\renewcommand{\arraystretch}{1.5}
\centering
\par%
\begin{tabular}{p{34pt}p{25pt}p{25pt}p{25pt}p{25pt}p{30pt}}
\hline Port & Delivery (TEU) & Pickup (TEU) & Velocity (knots) & Arrival time & Bunkering Amount\\
\hline
Hongkong & & & & 0 & 605\\
Chiwan & 2506 & 1443 & 14.0 & 2 & \\
Shanghai & 2078 & 842 & 14.0 & 61 & \\
Xingang & & 437 & 14.0 & 85 & \\
Dalian & 199 & 1037 & 14.0 & 100 & \\
Qingdao & & 1241 & 14.0 & 119 & \\
Hongkong & & & 14.0 & 196 & \\
\hline
\end{tabular}
\end{table}

\begin{table}
\caption{Optimal routing and fuel strategy: Ship 5}%
\label{caseResult5}
\renewcommand{\arraystretch}{1.5}
\centering
\par%
\begin{tabular}{p{34pt}p{25pt}p{25pt}p{25pt}p{25pt}p{30pt}}
\hline Port & Delivery (TEU) & Pickup (TEU) & Velocity (knots) & Arrival time & Bunkering Amount\\
\hline
HongKong & & & & 0 & \\
Busan & 2023 & & 14.0 & 82 & \\
Yokohama & & 3392 & 14.0 & 128 & \\
Kaosiung & & 942 & 14.0 & 226 & 598 \\
Chiwan & & 666 & 14.0 & 252 & \\
Hongkong & & & 14.0 & 254 & \\
\hline
\end{tabular}
\end{table}

\section{Conclusions}
To improve financial and environmental sustainability, an integrated fleet deployment and bunker management model is developed for a shipping company providing shipping services through a deport. With the objective of maximizing total profit, the optimal solution of the model is able to provide a more profitable way of shipping service than the methods separately considering deployment and bunker management.

A more realistic fuel consumption estimation function is proposed and adopted in the integrated to further enhance the effectiveness in practice, in which only only cruise speed but freight tonnage onboard are also taken into account when estimating ship fuel consumption.

The corresponding optimization problem is NP-hard and is formulated as a Mix Integer Programming problem with nonlinear terms. Although the problem can be linearly approximated and solved using CPLEX, the complexity still leads to a very high computational time. A two-stage method with two possible generating algorithms is proposed, among which a time-efficient heuristic based on Most Promising Area is recommended. The algorithm is proven to converge to local optimum with probability 1, while the computation time deceased on a large extent. An improvement step is also proposed so that each solution in the two-stage algorithm can be examined and further optimized.

Future work is aiming at extending the current work to a stochastic planning model to prevent disruptions and delays. On-line control of a fleet is expected to be achieved in case that disruptions occur. The current literatures with stochastic models mainly address problems with uncertain demand, and most of them are analyzed based on liner shipping\cite{Ng14,Meng12}. The situation with disruptions in tramp mode should be studied.

\bibliographystyle{elsarticle-num}
\bibliography{TYY_Bib}

\begin{thebibliography}{10}
\expandafter\ifx\csname url\endcsname\relax
  \def\url#1{\texttt{#1}}\fi
\expandafter\ifx\csname urlprefix\endcsname\relax\def\urlprefix{URL }\fi
\expandafter\ifx\csname href\endcsname\relax
  \def\href#1#2{#2} \def\path#1{#1}\fi

\bibitem{Sambracos04}
P.~J. T. C. K.~C. Sambracos, E., Dispatching of small containers via coastal
  freight liners: The case of the aegean sea, European Journal of Operational
  Research 152~(2) (2004) 365--381.

\bibitem{Notteboom06}
T.~E. Notteboom, The time factor in liner shipping services, Maritime Economics
  \& Logistics 8~(1) (2006) 19--39.

\bibitem{Ronen11}
D.~Ronen, The effect of oil price on containership speed and fleet size, The
  Journal of the Operational Research Society 62~(1) (2011) 211--216.

\bibitem{Leach08}
P.~T. Leach, Galloping gulf line (2008).

\bibitem{PracticalShipDesign}
D.~G.~M. Watson, Practical ship design, Elsevier ocean engineering book series:
  v. 1, Amsterdam ; New York : Elsevier, 1998.

\bibitem{Ronen82}
D.~Ronen, The effect of oil price on the optimal speed of ships, The Journal of
  the Operational Research Society 33~(11) (1982) 1035--1040.

\bibitem{Christiansen07}
M.~Christiansen, K.~Fagerholt, B.~Nygreen, D.~Ronen, Chapter 4 Maritime
  Transportation, Vol. Volume 14, Elsevier, 2007, pp. 189--284.

\bibitem{Alvarez09}
J.~F. Alvarez, Joint routing and deployment of a fleet of container vessels,
  Maritime Economics \& Logistics 11~(2) (2009) 186--208.

\bibitem{Yao12}
Z.~Yao, S.~H. Ng, L.~H. Lee, A study on bunker fuel management for the shipping
  liner services, Computers \& Operations Research 39~(5) (2012) 1160--1172.

\bibitem{Nagy05}
G.~Nagy, S.~Salhi, Heuristic algorithms for single and multiple depot vehicle
  routing problems with pickups and deliveries, European Journal of Operational
  Research 162~(1) (2005) 126--141.

\bibitem{Min89}
H.~Min, The multiple vehicle routing problem with simultaneous delivery and
  pick-up points, Transportation Research Part A: General 23~(5) (1989)
  377--386.

\bibitem{Gendreau99}
M.~Gendreau, G.~Laporte, D.~Vigo, Heuristics for the traveling salesman problem
  with pickup and delivery, Computers \& Operations Research 26~(7) (1999)
  699--714.

\bibitem{Ganesh07}
K.~Ganesh, T.~T. Narendran, Cloves: A cluster-and-search heuristic to solve the
  vehicle routing problem with delivery and pick-up, European Journal of
  Operational Research 178~(3) (2007) 699--717.

\bibitem{Nowak08}
M.~Nowak, O.~Ergun, C.~C. White, Pickup and delivery with split loads,
  Transportation Science 42~(1) (2008) 32--43.

\bibitem{Weng95}
Z.~K. Weng, Channel coordination and quantity discounts, Management Science
  41~(9) (1995) 1509--1509.

\bibitem{HongNelson06}
L.~Hong, B.~L. Nelson, Discrete optimization via simulation using compass,
  Operations Research 54~(1) (2006) 115--129.

\bibitem{Ng14}
M.~Ng, Distribution-free vessel deployment for liner shipping, European Journal
  of Operational Research 238~(3) (2014) 858--862.

\bibitem{Meng12}
W.~T. . W.~S. Meng, Q., Short-term liner ship fleet planning with container
  transshipment and uncertain container shipment demand, European Journal of
  Operational Research 223~(1) (2012) 96--105.

\end{thebibliography}

\end{document}